\documentclass[a4paper,12pt]{article}
\usepackage{amssymb,amsmath,amsthm,latexsym}
\usepackage{amsfonts}
	\usepackage{amsfonts}
\usepackage{graphicx}
\usepackage[pdftex,bookmarks,colorlinks=false]{hyperref}
\usepackage{verbatim}
\usepackage{caption}
\usepackage{subcaption}
\usepackage[hmargin=1.2in,vmargin=1.2in]{geometry}

\newtheorem{theorem}{Theorem}[section]

\newtheorem{definition}[theorem]{Definition}

\newtheorem{problem}{Problem}
\newtheorem{proposition}[theorem]{Proposition}
\newtheorem{remark}[theorem]{Remark}

\newtheorem{illustration}[theorem]{Illustration}

\title{\bf \sc A Study on Prime Arithmetic Integer Additive Set-Indexers of Graphs}
\author{{\bf N K Sudev} \\{\small Department of Mathematics, Vidya Academy of Science \& Technology} \\ {\small Thalakkottukara, Thrissur - 680501, Kerala, India.}\\ {\small email: {\em sudevnk@gmail.com}}\\ \\ {\bf K A Germina} \\{\small Department of Mathematics, School of Mathematical \& Physical Sciences} \\ {\small Central University of Kerala, Kasaragod - 671316, Kerala, India.}\\ {\small email: {\em srgerminka@gmail.com}}}

\date{}
\begin{document}
\maketitle

\begin{abstract}
 Let $\mathbb{N}_0$ be the set of all non-negative integers and $\mathcal{P}(\mathbb{N}_0)$ be its power set. An integer additive set-indexer (IASI) is defined as an injective function $f:V(G)\to \mathcal{P}(\mathbb{N}_0)$ such that the induced function $f^+:E(G) \to \mathcal{P}(\mathbb{N}_0)$ defined by $f^+(uv) = f(u)+ f(v)$ is also injective, where $\mathbb{N}_0$ is the set of all non-negative integers. A graph $G$ which admits an IASI is called an IASI graph. An IASI of a graph $G$ is said to be an arithmetic IASI if the elements of the set-labels of all vertices and edges of $G$ are in arithmetic progressions. In this paper, we discuss about a particular type of arithmetic IASI called prime arithmetic IASI. 
\end{abstract}

\section{Introduction}
For all  terms and definitions, not defined in this paper, we refer to \cite{FH} and for more about graph labeling, we refer to \cite{JAG}. Unless mentioned otherwise, all graphs considered here are simple, finite and have no isolated vertices.

Let $\mathbb{N}_0$ denote the set of all non-negative integers. For all $A, B \subseteq \mathbb{N}_0$, the sum set of these sets, denoted by  $A+B$, is defined by $A + B = \{a+b: a \in A, b \in B\}$. If either $A$ or $B$ is infinite, then $A+B$ is also infinite. Hence, all sets mentioned in this paper are finite sets of non-negative integers. We denote the cardinality of a set $A$ by $|A|$.

An {\em integer additive set-indexer} (IASI, in short) is defined in \cite{GA} as an injective function $f:V(G)\to \mathcal{P}(\mathbb{N}_0)$ such that the induced function $f^{+}:E(G) \rightarrow \mathcal{P}(\mathbb{N}_0)$ defined by $f^{+} (uv) = f(u)+ f(v)$ is also injective.  A graph $G$ which admits an IASI is called an IASI graph.  An IASI is said to be {\em $k$-uniform} if $|f^{+}(e)| = k$ for all $e\in E(G)$. That is, a connected graph $G$ is said to have a $k$-uniform IASI if all of its edges have the same set-indexing number $k$. The cardinality of the labeling set of an element (vertex or edge) of a graph $G$ is called the {\em set-indexing number} of that element. The vertex set $V$ of a graph $G$ is defined to be {\em $l$-uniformly set-indexed}, if all the vertices of $G$ have the same set-indexing number $l$. An element of $G$ whose set-indexing number $1$ is called a {\em mono-indexed element} of $G$.


%
%
By the term an AP-set, we mean a set whose elements are in arithmetic progression. In this paper, we study the characteristics given graphs, where the set-labels of whose vertices and edges are AP-sets. Since the elements of the set-labels of all elements of $G$ are in arithmetic progression, all set-labels must contain at least three elements. The common difference of the set-label of an element $x$ of a graph $G$ is called the {\em deterministic index} of $x$ and is denoted by $\varpi(x)$. The integer ratio between the deterministic indices of the end vertices of an edge $e$ in $G$ is called the {\em deterministic ratio} of $e$.

The following notions are introduced in \cite{GS7}.

Let $f:V(G)\to \mathcal{P}(\mathbb{N}_0)$ be an IASI on $G$. For any vertex $v$ of $G$, if $f(v)$ is an AP-set, then $f$ is called a {\em vertex-arithmetic IASI} of $G$. A graph that admits a vertex-arithmetic IASI is called a {\em vertex-arithmetic IASI graph}. For an IASI $f$ of $G$, if $f^+(e)$ is an AP-set, for all $e\in E(G)$, then $f$ is called an {\em edge-arithmetic IASI} of $G$. A graph that admits an edge-arithmetic IASI is called an {\em edge-arithmetic IASI graph}. An IASI is said to be an {\em arithmetic integer additive set-indexer} (AIASI, in short) if it is both vertex-arithmetic and edge-arithmetic. That is, an AIASI is an IASI $f$, under which the set-labels of all elements of a given graph $G$ are AP-sets. A graph that admits an AIASI is called an {\em AIASI graph}. 

\begin{theorem}\label{T-AIASI-g}
\cite{GS7} A graph $G$ admits an AIASI if and only if for any two adjacent vertices in $G$, the deterministic index of one vertex is a positive integral multiple of the deterministic index of the other vertex and this positive integer is less than or equal to the set-indexing number of the latter.
\end{theorem}

\noindent Let $v_i$ and $v_j$ are two adjacent vertices of $G$, with deterministic indices $d_i$ and $d_j$ respectively with respect to an IASI $f$ of $G$. Then, by Theorem \ref{T-AIASI-g}, $f$ is an arithmetic IASI if and only if $d_j=k\,d_i$, where $k$ is a positive integer such that $1\le k \le |f(v_i)|$ or equivalently, the deterministic ratio of the edge $uv$ is a positive integer less than or equal to the set-indexing number of the vertex having smaller deterministic index.

\begin{theorem}\label{T-AIASI1}
\cite{GS7} Let $G$ be a graph which admits an arithmetic IASI, say $f$ and let $d_i$ and $d_j$ be the deterministic indices of two adjacent vertices $v_i$ and $v_j$ in $G$. If $|f(v_i)|\ge |f(v_j)|$, then for some positive integer $1\le k\le |f(v_i)|$ , the edge $v_iv_j$ has the set-indexing number $|f(v_i)|+k(|f(v_j)|-1)$. 
\end{theorem}

In this paper, we introduce the notion of prime arithmetic IASI graphs and study the properties and characteristics of them.

\section{Prime AIASI Graphs}

\begin{definition}{\rm
A {\em prime arithmetic integer additive set-indexer} (prime AIASI) of a graph $G$ is an AIASI $f:V(G)\to \mathcal{P}(\mathbb{N}_0)$, where, for any two adjacent vertices in $G$ the deterministic index of one vertex is a prime integer multiple of the deterministic index of the other, where this prime integer is less than or equal to the set-indexing number of the second vertex. }
\end{definition}

In other words, an AIASI $f$ is a prime AIASI on $G$ if for any two adjacent vertices $v_i$ and $v_j$ of $G$ with the deterministic indices $d_i$ and $d_j$ respectively where $d_i\le d_j$, $d_j=p_id_i$ where $p_i$ is a prime integer such that $1< p_i\le |f(v_i)|$. 

In the following discussions, we discuss the admissibility of certain graphs. 

\begin{theorem}\label{T-PIASI-g}
A graph $G$ admits a prime AIASI if and only if it is bipartite. 
\end{theorem}
\begin{proof}
Let $G$ be a bipartite graph with a bipartition $(X,Y)$. Label all the vertices in $X$ by distinct AP-sets having the same common difference $d$ and label all the vertices of $Y$ by distinct AP-sets having the same common difference $d'$, where $d'=p\,d$, where $p=\min_{v\in X}|f(v)|$ is a prime integer. This set-labeling is a prime-AIASI defined on $G$. 

Conversely, assume that $G$ admits a prime AIASI, say $f$. Let $v$ be an arbitrary vertex of $G$. Let $V_1$ denote the set of all vertices that are adjacent to $v$. Since $G$ is a prime AIASI graph, for any vertex $v'$ in $V_1$, either $\varpi(v')=p_i\, \varpi(v);~ p_i \le |f(v)|$ (or $\varpi(v)=p_j\, \varpi(v);~ p_i \le |f(v')|$), where $p_i$ (or $p_j$) is a prime integer. Now, let $V_2$ be the neighbouring set of $V_1$. Since, $G$ admits a prime AIASI, no vertices in $V_2$ can be adjacent to $v$. Let $V_3$ be the neighbouring set of the vertices of $V_2$ and as explained above no vertices of $V_1$ will be adjacent to the vertices in $V_3$. Proceeding like this in finite number of times, say $m$, we cover all the vertices of $G$. Let $X=\cup_{i=0}^m V_{2i}$ and $Y=\cup_{i=0}^mV_{2i+1}$. Clearly, $(X,Y)$ is a bipartition on $G$.
\end{proof}

The following are some illustrations to prime AIASI graphs.

\begin{remark}{\rm 
Due to Theorem \ref{T-PIASI-g}, the paths, trees, even cycles and all acyclic graphs admit prime AIASIs.}
\end{remark}

\begin{illustration}
The subdivision graph of any non-trivial graph $G$ admits a prime AIASI graph.
\end{illustration} 
A subdivision graph, denoted by $G^{\ast}$, is the graph obtained by introducing a vertex to every edge of $G$. Then, $G^{\ast}$ is a bipartite graph. Hence, by Theorem \ref{T-PIASI-g}, $G^{\ast}$ admits a prime AIASI.

\begin{illustration}
Every median graph admits a prime AIASI.
\end{illustration}
A median graph $G$ is an undirected graph in which every three vertices $u,v$ and $w$ have a unique vertex $x$ that belongs to shortest paths between each pair of $u,v$,and $w$. Clearly, $G$ is a bipartite graph and hence by Theorem \ref{T-PIASI-g}, $G$ admits a prime-AIASI.

\begin{illustration}
Every hypercube graph $Q_n$ and partial cube graph admit a prime AIASI.
\end{illustration}
Every hypercube graph $Q_n$ is a Cartesian product of two bipartite graphs and hence it is a bipartite graph. Then, by Theorem \ref{T-PIASI-g}, $Q_n$ admits a prime-AIASI. A partial cube is a graph that is an isometric subgraph of a hypercube. Since it preserves distances, it is also a bipartite graph and hence admits a prime AIASI.

\section{Dispensing Number of Certain Graph Classes}

By Theorem \ref{T-PIASI-g}, a non-bipartite graph does not admit a prime AIASI. That is, some edges of a non-bipartite graph have non-prime deterministic ratio.  Then, we define the following notion.

\begin{definition}{\rm
The minimum possible number of edges in a graph $G$ that do not have a prime deterministic ratio is called the {\em dispensing number} of $G$ and is denoted by $\vartheta(G)$.}
\end{definition}

In other words, the dispensing number of a graph $G$ is the minimum number of edges to be removed from $G$ so that it admits a prime AIASI. Hence, we have

\begin{theorem}\label{T-PAIASI1}
If $b(G)$ is the number of edges in a maximal bipartite subgraph of a graph $G$, then $\vartheta(G)= |E(G)|-b(G)$.
\end{theorem}
\begin{proof}
Let $G$ be a non-bipartite graph which has $n$ vertices and $m$ edges. Let $f$ be an arithmetic IASI defined on $G$ which labels the vertices of $G$ in such a way that maximum number of edges have a prime deterministic ratio. This can be done by assigning distinct AP-sets to the vertices of $G$ in such a way that the deterministic index of a vertex (other than the vertex having the smallest deterministic index) is a prime multiple of the common difference of the set-label of some other vertex in $G$. Let $E_p$ be the set of all edges in $G$ which have a prime deterministic ratio in $G$ with respect to $f$. Then, clearly the induced subgraph  $\langle E_p \rangle$ of $G$ is a maximal bipartite subgraph of $G$. The edges in $G-\langle E_p \rangle$ do not have a prime deterministic ratio. Therefore, $\vartheta(G)=|G-\langle E_p \rangle| =|E(G)|-|\langle E_p \rangle|=|E(G)|-b(G)$.
\end{proof}

Invoking Theorem \ref{T-PAIASI1}, we investigate the sparing number about the dispensing number of certain graph classes.

\begin{proposition}\label{P-PAIASI2}
The dispensing number of an odd cycle is $1$.
\end{proposition}
\begin{proof}
Let $C_n$ be an odd cycle. Since no bipartite graphs contain odd cycles, $C_n$ is not a bipartite graphs. Let $e$ be any edge of $C_n$. Then, $C_n-e$ is a path on $n$ vertices, which is a bipartite graph. Therefore, for odd $n$, $\vartheta(C_n)=1$.
\end{proof}

\begin{theorem}\label{T-PAIASI-3}
The dispensing number of a complete graph $K_n$ is
\begin{equation*}
\vartheta(K_n)=
\begin{cases}
\frac{1}{4}n(n-2) &  \text{if $n$ is even}\\
\frac{1}{4}(n-1)^2 &  \text{if $n$ is odd}.
\end{cases}
\end{equation*} 
\end{theorem}
\begin{proof}
The maximal bipartite subgraph of $K_n$ is a complete bipartite graph $K_{a,b}$, where $a+b=n$. The number of edges in $K_n-K_{a,b}$ is $\frac{1}{2}(a+b)(a+b-1)-ab=\frac{1}{2}(a^2+b^2-a-b)=\frac{1}{2}[a(a-1)+b(b-1)]$. Then, we have the following cases.

\noindent {\em Case-1:} Let $n$ be an even integer. Then, the number of edges in $K_{a,b}$ is maximum if $a=b=\frac{n}{2}$. Hence, by Theorem \ref{T-PIASI-g}, the dispensing number of $K_n$ is $\frac{1}{2}[\frac{n}{2}(\frac{n}{2}-1)+\frac{n}{2}(\frac{n}{2}-1)]=\frac{1}{4}n(n-2)$. 

\noindent {\em Case-2:} Let $n$ be an odd integer. Then, the number of edges in $K_{a,b}$ is maximum if $a=\lfloor \frac{n}{2}\rfloor= \frac{n+1}{2}$ and $b=\lceil \frac{n}{2} \rceil =\frac{n+1}{2}$ for odd $n$. Then, by Theorem \ref{T-PIASI-g}, the dispensing number of $K_n$ is $\frac{1}{2}[\frac{n-1}{2}(\frac{n-1}{2}-1)+\frac{n+1}{2}(\frac{n+1}{2}-1)]=\frac{1}{4}(n-1)^2$. This completes the proof.
\end{proof}

The union of two graphs $G_1$ and $G_2$, the {\em union} of $G_1\cup G_2$ is the graph whose vertex set is $V(G_1)\cup V(G_2)$ and the edge set $E(G_1)\cup E(G_2)$. If two graphs $G_1$ and $G_2$ have no common elements, then their union is said to be the {\em disjoint union}. The following theorem estimates the dispensing number of the union of two graphs. 

Another interesting graph whose vertex set has two partitions is a split graph. A {\em split graph}  is a graph whose vertex has two partitions of which one is an independent set, say $S$ and the subgraph induced by the other is a block, say $K_r$. A split graph is said to be a {\em complete split graph} if every vertex of the independent set $S$ is adjacent to all vertices of the block $K_r$. 

Invoking Theorem \ref{T-PAIASI-3}, we estimate the dispensing number of a split garaph and complete split graph in the following theorem.

\begin{theorem}
Let $G$ be a split graph with a block $K_r$ and an independent set $S$, where $K_r\cup \langle S \rangle$. Then,
\begin{equation*}
\vartheta(G)=
\begin{cases}
\frac{1}{4}r(r-2) &  \text{if $r$ is even}\\
\frac{1}{4}(r-1)^2 &  \text{if $r$ is odd}.
\end{cases}
\end{equation*}
\end{theorem}
\begin{proof}
Since $S$ is an independent set in $G$, the maximal bipartite subgraph of $G$ is obtained by eliminating required number of edges from the block $K_r$ only. Since $K_r$ is a complete graph in $G$, by Theorem \ref{T-PAIASI-3}, the minimum number of edges to be removed from $K_r$ so that it becomes a bipartite graph is $\frac{1}{4}r(r-2)$ if $r$ even and $\frac{1}{4}(r-1)^2$ if $r$ is even.
Therefore, 
\begin{equation*}
\vartheta(G)=
\begin{cases}
\frac{1}{4}r(r-2) &  \text{if $r$ is even}\\
\frac{1}{4}(r-1)^2 &  \text{if $r$ is odd}.
\end{cases}
\end{equation*}
\end{proof}

\noindent Next, we proceed to determine the dispensing number of the union of two graphs.

\begin{theorem}\label{T-PAISAI4}
Let $G_1$ and $G_2$ be two given IASI graphs. Then, $\vartheta(G_1\cup G_2)=\vartheta(G_1)+\vartheta(G_2)-\vartheta(G_1\cap G_2)$. In particular, if $G_1$ and $G_2$ are two edge-disjoint graphs, then $\vartheta(G_1\cup G_2)=\vartheta(G_1)+\vartheta(G_2)$. 
\end{theorem} 
\begin{proof}
Let $G_1$ and $G_2$ admit arithmetic IASIs, say $f_1$ and $f_2$ respectively such that $f_1(v_i)\ne f_2(w_j)$ for any two distinct vertices $v_i\in V(G_1)$ and $w_j\in V(G_2)$. Define a function $f: V(G_1\cup G_2)\to \mathcal{P}(\mathbb{N}_0)$ as follows.
\begin{equation*}
f(v)=
\begin{cases}
f_1(v) & \text{if}~~ v\in V(G_1)\\
f_2(v) & \text{if}~~ v\in V(G_2).
\end{cases}
\end{equation*}

If $G_1$ and $G_2$ are edge-disjoint, no edge in $G_1$ has the same set-label of a vertex in $G_2$. Hence $f$ is injective. Therefore, $\vartheta(G_1\cup G_2)=\vartheta(G_1)+\vartheta(G_2)$.

If $G_1$ and $G_2$ have some edges in common, then $G_1\cap G_2$ is the graph induced by these edges. Then, $f_1=f_2$ for all the vertices in $G_1\cap G_2$.  Therefore, the graphs $G_1-G_1\cap G_2$, $G_2-G_1\cap G_2$ and $G_1\cap G_2$ are disjoint graphs whose union is the graph $G_1\cup G_2$. Hence, $G_1\cup G_2 =(G_1-G_1\cap G_2)\cup (G_2-G_1\cap G_2)\cup (G_1\cap G_2)$. Then, by the above argument, 
$\vartheta(G_1\cup G_2) =\vartheta(G_1-G_1\cap G_2)+ \vartheta(G_2-G_1\cap G_2)+ \vartheta(G_1\cap G_2)=\vartheta(G_1)+\vartheta(G_2)-\vartheta(G_1\cap G_2)$. This completes the proof.
\end{proof}

Invoking Theorem \ref{T-PAISAI4}, we establish the following theorem on the dispensing number of Eulerian graphs.

\begin{theorem}
The dispensing number of an Eulerian graph $G$ is the number of odd cycles in $G$.
\end{theorem}
\begin{proof}
Let $G$ be an Eulerian graph. Then, $G$ can be decomposed into a finite number of edge disjoint cycles. Let $C_{n_1},C_{n_2},C_{n_3},\ldots, C_{n_l}$ be the edge disjoint cycles in $G$ such that $G=\cup_{i=1}^lC_{n_i}$. Without loss of generality let $C_{n_1}, C_{n_2},C_{n_3}\ldots C_{n_r}$ are odd cycles $C_{n_{r+1}},C_{n_{r+2}},\ldots, C_{n_l}$ are even cycles. By Theorem \ref{T-PIASI-g}, the dispensing number of even cycles is $0$ and by Proposition \ref{P-PAIASI2}, the dispensing number of odd cycles is $1$. Since The cycles in $G$ are edge disjoint, all odd cycles in $G$ has one edge that has no prime deterministic ratio. Therefore, by Theorem \ref{T-PAISAI4}, the dispensing number of $G$ is equal to the number of odd cycles in $G$. 
\end{proof}

A {\em catus} $G$ is a connected graph in which any two simple cycles have at most one vertex in common. That is, every edge in a cactus belongs to at most one simple cycle. The cycles in a cactus are edge disjoint. Hence, we have

\begin{theorem}
The dispensing number of a cactus $G$ is the number of odd cycles in $G$.
\end{theorem} 
\begin{proof}
Let $G$ be a cactus. Let $C_{n_1},C_{n_2},C_{n_3},\ldots, C_{n_l}$ be cycles in $G$. All these cycles are edge disjoint. Without loss of generality let $C_{n_1}, C_{n_2},C_{n_3}\ldots C_{n_r}$ are odd cycles $C_{n_{r+1}},C_{n_{r+2}},\ldots, C_{n_l}$ are even cycles. By Theorem \ref{T-PIASI-g}, the dispensing number of even cycles is $0$ and by Proposition \ref{P-PAIASI2}, the dispensing number of odd cycles is $1$. Then, all odd cycles in $G$ has one edge that has no prime deterministic ratio. Therefore, by Theorem \ref{T-PAISAI4}, the dispensing number of $G$ is equal to the number of odd cycles in $G$. 
\end{proof}

Another graph class similar to cactus graphs is a {\em block graph} (or a {\em clique tree}) which is defined as an undirected graph in which every biconnected component (block) is a clique. The following theorem estimates the dispensing number of a block graph

\begin{theorem}
The dispensing number of a block graph $G$ is the sum of the dispensing numbers of the cliques in $G$.
\end{theorem}
\begin{proof}
Let $K_{n_1},K_{n_2},K_{n_3},\ldots K_{n_l}$ be the cliques in $G$. Since $G$ is a block graph, then any two cliques in it hs at most one common vertex. That is, all cliques in $G$ are edge disjoint. then, by Theorem \ref{T-PAISAI4}, $\vartheta(G)=\vartheta(\cup_{i=1}^l K_{n_i}) = \sum_{i=1}^{l}\vartheta(K_{n_i})$.
\end{proof}

\section{Conclusion}

In this paper, we have discussed some characteristics of graphs which admit a certain types of IASIs called prime arithmetic IASI. We have established some conditions for some graph classes to admit this types of arithmetic IASIs and have explained certain properties and characteristics of prime arithmetic IASI graphs. Some of the problems in this area we have identified are the following.

\begin{problem}{\rm A cactus is a subclass of outerplanar graphs. An {\em outerplanar graph} is a planar graph all of whose vertices belong to the exterior unbounded plane of the graph. In an outer planar, two simple cycles can have at most one edge in common. Finding the dispensing number of outer planar graphs is an open problem.}
\end{problem}

\begin{problem}{\rm Estimating the dispensing number of graph operations such as graph joins, graph complements etc.is also an open problem.}
\end{problem}

\begin{problem}{\rm The problem of finding the dispensing number of various graph products such as Cartesian products, strong products, lexicographic products, corona etc. is  open.}
\end{problem}

\begin{problem}{\rm The problem of finding the dispensing number of various graph classes such as bisplit graphs, sun graphs etc. is  open.}
\end{problem}

Problems related to the characterisation of different arithmetic IASI graphs are still open. The IASIs under which the vertices of a given graph are labeled by different standard sequences of non negative integers, are also worth studying.   The problems of establishing the necessary and sufficient conditions for various graphs and graph classes to admit certain IASIs still remain unsettled. All these facts highlight a wide scope for further studies in this area.


\begin{thebibliography}{15}
\bibitem {A10} B D Acharya, (1990). {\em Arithmetic Graphs}, J. Graph Theory, {\bf 14}(3), 275-299. 
\bibitem {AGA} B D Acharya, K A Germina and T M K Anandavally, {\em Some New Perspective on Arithmetic Graphs} In {\bf Labeling of Discrete Structures and Applications}, (Eds.: B D Acharya, S Arumugam and A Rosa), Narosa Publishing House, New Delhi, (2008), 41-46.
\bibitem {TMA} Tom M Apostol, {\bf Introduction to Analytic Number Theory}, Springer-Verlag, New York, (1989).
\bibitem {BM1} J A Bondy and U S R Murty, (2008). {\bf Graph Theory}, Springer.
\bibitem {BLS} A Brandst\"{a}dt, V B Le and J P Spinard, (1999). {\bf Graph Classes:A Survey}, SIAM, Philadelphia.
\bibitem {CZ} G Chartrand and P Zhang, (2005). {\bf Introduction to Graph Theory}, McGraw-Hill Inc.
\bibitem {ND} N Deo, (1974). {\bf Graph Theory with Applications to Engineering and Computer Science}, PHI Learning.
\bibitem {JAG} J A Gallian, (2011). {\em A Dynamic Survey of Graph Labelling}, The Electronic Journal of Combinatorics (DS 16).
\bibitem {GA} K A Germina and T M K Anandavally, (2012). {\em Integer Additive Set-Indexers of a Graph:Sum Square Graphs}, Journal of Combinatorics, Information and System Sciences, {\bf 37}(2-4), 345-358.
\bibitem {GS1} K A Germina, N K Sudev, (2013). {\em On Weakly Uniform Integer Additive Set-Indexers of Graphs}, Int. Math. Forum., {\bf 8}(37), 1827-1834.
\bibitem {GS2} K A Germina, N K Sudev, {\em Some New Results on Strong Integer Additive Set-Indexers}, Communicated.
\bibitem {GY} J. Gross, J. Yellen, {\bf Graph Theory and Its Applications}, CRC Press, (1999).
\bibitem {FH}  F Harary, (1969). {\bf Graph Theory}, Addison-Wesley Publishing Company Inc.
\bibitem {KDJ} K D Joshi, {\bf Applied Discrete Structures}, New Age International, (2003).
\bibitem {SMH} S M Hegde, (1989). {\em Numbered Graphs and Their Applications}, PhD Thesis, Delhi University.
\bibitem {MBN} M B Nathanson (1996). {\bf Additive Number Theory, Inverse Problems and Geometry of Sumsets}, Springer, New York.
\bibitem {GS0} N K Sudev and K A Germina, (2014). {\em A Note on Integer Additive Set-Indexers of Graphs}, Int. J. Math. Sci.\& Engg. Applications, {\bf 8}(2), 11-22.
\bibitem {GS7} N K Sudev and K A Germina, {\em On Arithmetic Integer Additive Set-Indexers of Graphs}, submitted.
\bibitem {GS8} N K Sudev and K A Germina, {\em On Certain Types of Arithmetic Integer Additive Set-Indexers of Graphs}, submitted.
\bibitem {GS9} N K Sudev and K A Germina, {\em A Study on Semi-Arithmetic Integer Additive Set-Indexers of Graphs}, Int. J. Math. Sci. \& Engg. Applns. {\bf 8}(3), 157-165.
\bibitem {DBW} D B West, (2001). {\bf Introduction to Graph Theory}, Pearson Education Inc.
\end{thebibliography}
\end{document}